\documentclass[12pt]{amsart}
\usepackage{graphicx}
\usepackage{amsfonts,amssymb,amscd,amsmath,enumerate,verbatim,calc}

\newtheorem{Theorem}{Theorem}[section]

\newtheorem{Corollary}[Theorem]{Corollary}
\newtheorem{Proposition}[Theorem]{Proposition}
\newtheorem{Remark}[Theorem]{Remark}

\newtheorem{Definition}[Theorem]{Definition}

\begin{document}
\title{ON SOME AUTOMORPHISM RELATED PARAMETERS IN GRAPHS}
\author{Imran Javaid$^*$, Hira Benish, Usman Ali, M. Murtaza}
\keywords{fixing number, fixing share, fixing percentage. 
\indent $^*$ Corresponding author: imran.javaid@bzu.edu.pk}
\address{Centre for Advanced Studies in Pure and Applied Mathematics,
Bahauddin Zakariya University Multan, Pakistan\newline Email:
imran.javaid@bzu.edu.pk, hira\_benish@yahoo.com,
uali@bzu.edu.pk,\newline mahru830@gmail.com}


\date{}
\maketitle
\begin{abstract}
In this paper, we deduce some properties of $f$-sets of connected
graphs. Also, we introduce the concept of fixing share of each
vertex of a fixing set $D$ to see the participation of each vertex
when fixing a connected graph $G$. We define a parameter, called the
fixing percentage, by using the concept of fixing share, which is
helpful in determining the measure of the amount of fixing done by
the elements of $D$ in $G$.
\end{abstract}
\section{Preliminaries}
Let $G$ be a graph with the vertex set $V(G)$ and the edge set
$E(G)$. If two vertices $u$ and $v$ share an edge, then they are
called \emph{adjacent}, otherwise they are called
\emph{non-adjacent}. The open neighborhood of a vertex $u$ is
$N(u)=\{v\in V(G):$ $v$ is adjacent to $u$ in $G\}$, and the closed
neighborhood of $u$ is $N[u] = N(u)\cup \{u\}$. For a subset $U$ of
$V(G)$, the set $N_G(U)=\{v \in V(G):v$ is adjacent to some $u\in
U\}$ is the open neighborhood of $U$ in $G$.
 Two distinct vertices $u,v$ are adjacent twins if $N[u]=N[v]$ and non-adjacent twins if $N(u)=N(v)$. A set $U\subseteq V(G)$ is called
a twin-set of $G$ if $u, v$ are twins in $G$ for every pair of
distinct vertices $u,v\in U$. The \emph{distance} $d(u,v)$ between
two vertices $u,v \in V(G)$ is the shortest length of a path between
them and the diameter $diam(G)$ of $G$ is $\max \limits_{u,v \in
V(G)}d(u,v)$. We refer to the book \cite{1} for the general graph
theoretic notation and terminology not described in this paper.

For a graph $G$, an {\it automorphism} of $G$ is a bijective mapping
$f:V(G)\rightarrow V(G)$ such that $f(u)f(v)\in E(G)$ if and only if
$uv\in E(G)$. The set of all automorphisms of $G$ forms a group,
denoted by $\Gamma(G)$, under the operation of composition of
functions. For a vertex $v$ of $G$, the set $\{f(v):f\in
\Gamma(G)\}$ is the {\it orbit} of $v$, denoted by $\mathcal{O}(v)$,
and any two vertices in the same orbit are called {\it similar
vertices}. If $u$ and $v$ are similar, then we write $u\sim v$,
otherwise we write $u\not \sim v$. Let us consider sets $S(G)=\{v\in
V(G):$ $|\mathcal{O}(v)|\ge 2 \}$ and $V_s(G)=\{(u,v):$ $u\neq v$
and $\mathcal{O}(u)=\mathcal{O}(v)\}$. Throughout the paper, the
notation $(u,v)$ represents unordered pair of vertices because the
relation `$\sim$' is symmetric.
 Also, if $G$ is a rigid graph (i.e., a graph with $\Gamma (G)={id}$), then $V_s(G)=\emptyset$. Recall that every
automorphism is also an {\it isometry},
that is, for $u,v\in V(G)$ and $g\in \Gamma(G)$, $d(u,v)=d(g(u),g(v))$.\\

 An automorphism $g\in \Gamma(G)$ is said to \emph{fix} a vertex $v\in V(G)$ if $g(v)=v$. The set of automorphisms
that fix a vertex $v\in V(G)$, called the {\it stabilizer} of $v$,
is a subgroup of $\Gamma(G)$ denoted as $\Gamma_v(G)$. An
automorphism $g\in \Gamma(G)$ is said to {\it fix} a set $D\subseteq
V(G)$ if for every $v\in D$, we have $g(v)=v$. The set of
automorphisms that fix $D$ is a subgroup $\Gamma_D(G)$ of
$\Gamma(G)$ and $\Gamma_D(G)=\cap_{v\in D}\Gamma_v(G)$. If $D$ is a
set of vertices for which $\Gamma_D(G)=\{id\}$, then $D$ {\it fixes}
the graph $G$ and we say that $D$ is a {\it fixing set} of $G$.
Erwin and Harary introduced the {\it fixing number}, $fix(G)$, of a
graph $G$ in \cite{erw} and it is defined as the minimum cardinality
of a set of vertices that fixes $G$. A fixing set containing
$fix(G)$ vertices is called a \emph{minimum fixing set} of $G$.

A vertex $x\in V(G)$ is called a \emph{fixed} vertex if $g(x)=x$ for
all $g\in \Gamma(G)$, i.e., $\Gamma_x(G)=\Gamma(G)$. A vertex $x\in
V(G)$ is said to \emph{fix a pair} $(u,v)\in V_s(G)$, if $h(u)\neq
v$ or $h(v)\neq u$ whenever $h\in \Gamma_x(G)$. Note that fixing a
pair $(u,v)\not \in V_s(G)$ has no sense. Let $(u,v)\in V_s(G)$ and
the set $fix(u,v)=\{x\in V(G):$ $x$ fixes $(u,v)\}$ is called the
\emph{fixing set} (or \emph{$f$-set}) relative to the pair $(u,v)$.
It is also further assumed that if $(u,v)\not\in V_s(G)$, then
$fix(u,v)=\emptyset$. Thus, $\{u,v\}\subseteq fix(u,v)\subseteq
V(G)$. Let $x\in V(G)$ and the set $F(x)=\{(u,v)\in V_s(G) : $ $x$
fixes $(u,v) \}$ is called the {\it fixed neighborhood} of $x$.
Also, if $x\in V(G)$ is a fixed vertex, then $F(x)=\emptyset$. The
{\it fixing graph} $F(G)$ is a bipartite graph with bipartition
$(S(G),V_s(G))$ and a vertex $x\in S(G)$ is adjacent to a pair
$(u,v)\in V_s(G)$ if $x$ fixes $(u,v)$. Observe that
$N_{F(G)}(D)=\{(u,v)\in V_s(G):$ $x$ fixes $(u,v)$ for some $x\in
D\}$ for any set $D\subseteq S(G)$. Hence, the fixing number of $G$
is the minimum cardinality of a subset $D\subseteq S(G)$ such that
$N_{F(G)}(D)=V_s(G)$.
\par

An upper bound on $fix(G)$ was given by Erwin and Harary by using
another well-studied invariant, metric dimension, defined in the
following way. Let $W = \{v_1, v_2, \ldots, v_k\}$ be a $k$-subset
of $V(G)$ and, for each vertex $v \in V(G)$, define $r(v|W) = (d(v,
v_1), d(v, v_2), \ldots, d(v, v_k))$. A $k$-set $W$ is called a
resolving set for $G$ if for every pair $u, v$ of distinct vertices
of $G$, $r(u|W)\neq r(v|W)$. The { \it metric dimension} $dim(G)$ is
the smallest cardinality of a resolving set of $G$. A resolving set
of minimum cardinality is a \emph{metric basis} for $G$. The
following results were given in \cite{erw}.
\begin{Theorem}\cite{erw}
(i)If $W$ is a metric basis for $G$, then $\Gamma_W(G)$ is
trivial.\\
(ii) For every connected graph $G$, $fix(G)\leq dim(G)$.
\end{Theorem}

These last results establish metric dimension and fixing number are
closely related notions. Indeed, C\'{a}ceres et al. \cite{cac}, and
subsequently Garijo et at. \cite{gar}, dealt with the difference
between these parameters by studying the following question that
appeared first in \cite{bou}: Can the difference between both
parameters of a graph of order $n$ be arbitrarily large? \par
Another graph invariant related to metric dimension is the
\emph{resolving number} $res(G)$, which is the minimum $k$ such that
every $k$-set of vertices is a resolving set of a graph $G$. Observe
that this parameter is a natural upper bound on the metric dimension
of $G$: $dim(G)\leq res(G)$.

\par In the next section, we study some
properties of $f$-sets following the study of $R$-sets by Tomescu
and Imran \cite{tom}. To see the contribution of each vertex when
resolving a graph, the concepts of resolving share and resolving
percentage were introduced in \cite{jav}. In the second section, we
introduce the concept of {\it fixing share}, which tells about the
participation of each vertex of a fixing set when fixing a graph
$G$. We also define the {\it fixing percentage} in $G$, by using the
concept of fixing share of each element of a fixing set of $G$,
which is the measure of the amount of fixing done by a fixing set in
$G$. Then we compute the fixing share and the fixing percentage in
paths and cycles.

\section{Properties of $f$-sets}
\begin{Proposition}\label{l1}
If there exists an automorphism $g\in \Gamma(G)$ such that $g(u)=v,
u\neq v$ and if $d(u,x)= d(v,x)$ for some $x\in V(G)$, then $x\notin
fix(u,v)$.
\end{Proposition}
\begin{Proposition}\label{l2}
Let $G$ be a cycle of order $n$ and let $(u,v)\in V_s(G)$. \\
$(i)$ If $n$ is even and $d(u,v)$ is odd, then $fix(u,v)=V(G)$.\\
$(ii)$ If $n$ is even and $d(u,v)$ is even, then
$fix(u,v)=V(G)\setminus \{x_1,x_2\}$ where  $x_1,x_2 \in V(G)$ are
the only antipodal vertices with $d(x_i,u)= d(x_i,v)$, $1\leq i\leq 2$.\\
$(iii)$ If $n$ is odd and $x\in V(G)$ is the vertex with
$d(x,u)=d(x,v)$, then $fix(u,v)=V(G)\setminus \{x\}$.
\end{Proposition}
\begin{Proposition}\label{p1}
Let $G$ be a path of order $n$ and $V(G)=\{u_1,...,u_n\}$ where
$u_i$ is adjacent to $u_{i+1}$ with $(1\leq i\leq n-1)\}$, then
$V_s(G)=\{(u_i,u_{n+1-i}):1\leq i\leq \lfloor \frac{n}{2}\rfloor\}$ and \\
(i) If $n$ is even, then $fix(u_i,u_{n+1-i})=V(G)$.\\
(ii) If $n$ is odd, then $fix(u_i,u_{n+1-i})=V(G)\setminus
\{u_{\frac{n+1}{2}}\}$.
\end{Proposition}

Two vertices $u$ and $v$ in a graph $G$ are said to be \emph{twin}
if $d(u,w)=d(v,w)$ for all $w\in V(G)\backslash \{u,v\}$.
\begin{Proposition}\label{l1}
$fix(u,v)=\{u,v\}$ if and only if $u$ and $v$ are twin vertices.
\end{Proposition}
\begin{proof}
Let $fix(u,v)=\{u,v\}$. Now, if $u$ and $v$ are not twin, then there
exists a vertex $w\in V(G)$ such that $d(u,w)\neq d(v,w)$. If we fix
$w$ by an automorphism $f\in \Gamma_w(G)$, then $f(u)=v$ implies
$d(w,f(u))=d(w,v)$. Now $w=f(w)$ implies $d(f(w),f(u))=d(w,v)$ and
fact that $f$ is an isometry implies $d(w,u)=d(w,v)$, a
contradiction. Thus $u,v$ are twin.
\par
 Conversely, let $u$ and $v$ be two twin vertices and
$\{u,v\}\subset fix(u,v)$, then there exists at least one vertex
$w(\neq u,v)\in V(G)$ such that $f(u)\neq v$ and $f(v)\neq u$ for
all $f\in \Gamma_w(G)$. Since $f$ is an isometry and $f(u)\neq v$,
so $d(u,w)= d(f(u),f(w))\neq d(v,w)$, a contradiction that $u$ and
$v$ are twin.
\end{proof} In a complete graph, every pair of vertices are twin.
Therefore, we have the following corollary:
\begin{Corollary}\label{l2}
Let $G$ be a complete graph of order $n$ and $u$ and $v$ be a pair
of distinct vertices, then $(u,v)\in V_s(G)$ and $fix(u,v)=\{u,v\}$.
\end{Corollary}

\begin{Corollary}\label{l2}
Let $G=K_{m,n}$ be a complete bipartite graph, \\
(i) $fix(u,v)=\{u,v\}$, if both $u$ and $v$ are in same partite sets.\\
(ii) $fix(u,v)=V(G)$, if $u$ and $v$ are not in same partite sets and $m=n$.\\
(iii) $fix(u,v)=\emptyset$, if both $u$ and $v$ are not in same
partite set and $m\neq n.$
\end{Corollary}

\section{The fixing share in graphs}
\begin{Remark}\label{c01}
If $u$ and $v$ are twins in a connected graph $G$ and $D$ is a
fixing set for $G$, then $u$ or $v$ is in $D$. Moreover, if $u\in D$
and $v\notin D$, then $(D\setminus \{u\})\cup \{v\}$ is also a
fixing set for $G$.
\end{Remark}
\begin{Remark}\label{r01}
If $U$ is a twin-set in a connected graph $G$ of order $n$ with
$|U|= m\geq 2$, then every fixing set for $G$ contains at least
$m-1$ vertices from $U$.
\end{Remark}
\begin{Definition}\label{d31}
(Sole fixer) Let $G$ be a connected graph and $D$ be a minimum
fixing set of $G$. Let $(u,v)\in V_s(G)$. If $fix(u,v)\cap D=\{x\}$,
then $x$ is called sole fixer for the pair $(u,v)$.
\end{Definition}
\begin{figure}[h]
      {\includegraphics[width=8cm]{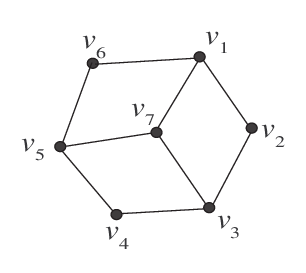}}
        \caption{The graph $G_1$}\label{fig2}
\end{figure}
Consider the graph $G_1$ in \textsc{Figure} 1 with vertex set
$V(G_1)=\{v_1,v_2,v_3,v_4,v_5,v_6,v_7\}$ and
$V_s(G_1)=\{(v_1,v_3),(v_1,v_5),(v_3,v_5),(v_2,v_4),(v_2,v_6),(v_4,v_6)\}$.
A fixing set of $G_1$ with minimum cardinality is $D_1=\{v_1,v_3\}$.
Since $fix(v_2,v_4)=\{v_1,v_5\}$ and $fix(v_2,v_6)=\{v_3,v_5\}$, so
$v_1$ is the sole fixer of the pair $(v_2,v_4)$ and $v_3$ is the
sole fixer of the pair $(v_2,v_6)$ with respect to
$D_1=\{v_1,v_3\}$.

\begin{Definition}\label{r02}
Let $G$ be a connected graph and $D$ be a minimum fixing set of $G$.
Let $(u,v)\in V_s(G)$ and $x\in D$ be the sole fixer of $(u,v)$ with
respect to $D$. Then, $x$ is considered to has 1-share in fixing
pair $(u,v)$. Further, if $(u,v)$ is also fixed by, say $m<|D|$
other vertices of $D\setminus \{x\}$, then $x$ is considered to have
$\frac{1}{m+1}$ share in fixing the pair $(u,v)$.
\end{Definition}
For example in $G_1$, the vertex $v_1$ is considered to have 1-share
in fixing the pair $(v_2,v_4)$ with respect to $D_1$. However,
vertex $v_1$ is considered to have $\frac{1}{2}$ -share in fixing
the pair $(v_4,v_6)$, as $fix(v_4,v_6)\cap D_1=\{x_1,x_3\}$.

If $D$ is a fixing set with minimum cardinality for a connected
graph $G$ and $x\in D$, then the fixing share of $x$ in $D$ is
defined as a measure of the amount of fixing done by $x$ in $G$.
Formally, we have the following definition:
\begin{Definition}\label{d32}
(fixing share) Let $G$ be a connected graph and let $D$ be a minimum
fixing set of $G$. For any $x\in D$ and $(u,v)\in F(x)$, we define a
set $\mathcal{F}(u,v)=\{F(y): y\in D$ and $(u,v)\in F(y)\}$. The
fixing share of a vertex $x\in D$ in $D$ is defined as
\[f(x;D)=\sum \limits_{(u,v)\in F(x)}\frac{1}{|\mathcal{F}(u,v)|}\]

\end{Definition}
In the example of Figure \ref{fig2}, we have
$F(v_1)=\{(v_1,v_5),(v_1,v_3),(v_2,v_4),(v_4,v_6)\}$ and
$F(v_3)=\{(v_1,v_3),(v_2,v_6),(v_3,v_5),(v_4,v_6)\}$. Also
$\mathcal{F}(v_1,v_5)=\{F(v_1)\}$,
$\mathcal{F}(v_1,v_3)=\{F(v_1),F(v_3)\}$,
$\mathcal{F}(v_2,v_4)=\{F(v_1)\}$,
$\mathcal{F}(v_4,v_6)=\{F(v_1),F(v_3)\}$. Thus, $f(v_1;D_1)=
1+\frac{1}{2}+1+\frac{1}{2}=3$ and similarly $f(v_3;D_1)=
\frac{1}{2}+1+1+\frac{1}{2}=3$.
\begin{Definition}\label{d33}
(fixing sum and percentage) Let $D$ be a minimum fixing set for a
connected graph $G$. The fixing sum and the fixing percentage of $G$
with respect to $D$ are defined as $F_{sum}(D)= \sum\limits_{x\in
D}f(x;D)$ and $F\%(D)=\frac{|D|}{F_{sum}(D)}$ respectively.
\end{Definition}
For the graph $G_1$ of Figure \ref{fig2}, $F_{sum}(D_1)=6$ and
$F\%(D_1)=\frac{2}{6}=\frac{1}{3}$.
\begin{Theorem}\label{t32}
Let $G$ be a complete graph of order $n\ge 3$ and let $D$ be a
minimum fixing set of $G$. For every $v\in D$, it holds that
$f(v;D)=\frac{n}{2}$. Furthermore, $F_{sum}(D)=(n-1)\frac {n}{2}$
and $F\%(D)=\frac{2}{n}$.
\end{Theorem}
\begin{proof}
Let $V(G)=\{v_1,...,v_n\}$. Since $G$ is a complete graph, so
$V_s(G)=\{(v_i,v_j):i\neq j$ and $1\leq i,j\leq n\}$. Also, the
cardinality of a minimum fixing set of $G$ is $n-1$. Let
$D=\{v_1,v_2,...,v_{n-1}\}\subset V(G)$ be such a minimum set (note
that $D$ is a twin-set in $G$). Then, for each $i$ with $1\leq i
\leq n-1$, we have that $F(v_i)=\{(v_i,v_j): j\neq i$ and $1\le j\le
n\}$. It can be seen that for each $j$ with $j\neq i$ and $1\leq
j\leq n-1$, we obtain $|\mathcal{F}(v_i,v_j)|=2$ because $(v_i,v_j)$
appears in exactly two $F(v_i)$, being $i$ such that $1\leq i\leq
n-1$. Also, $|\mathcal{F}(v_i,v_n)|=1$ since $(v_i,v_n)$ appears in
exactly one $F(v_i)$, being $i$ such that $1\leq i \leq n-1$.
Therefore, $f(v_i;D)=(n-2)\frac {1}{2}+1=\frac {n}{2}$, and
consequently $F_{sum}(D)=(n-1)\frac {n}{2}$ and
$F\%(D)=\frac{2}{n}$.
\end{proof}

\begin{Theorem}\label{t33}
Let $G$ be a path of order $n\geq2$ and let $D=\{v\}$ be a minimum
fixing set of $G$. Then, $f(v;D)=\lfloor
\frac{n}{2}\rfloor=F_{sum}(D)$ and $F\%(D)=\frac{1}{\lfloor
\frac{n}{2}\rfloor}$.
\end{Theorem}
\begin{proof}
As $fix(G)=1$ and one of the end vertices of $G$ forms a fixing set,
say $D=\{v\}$. Also $|V_s(G)|=\lfloor \frac{n}{2}\rfloor$ and $v$ is
the sole fixer of the $\lfloor \frac{n}{2}\rfloor$ pairs of
$V_s(G)$. Thus, $|F(v)|=\lfloor \frac{n}{2}\rfloor$ and
$|\mathcal{F}(u,v)|=1$ for all $(u,v)\in F(v)$. Hence,
$f(v;D)=\lfloor \frac{n}{2}\rfloor=F_{sum}(D)$ and
$F\%(D)=\frac{1}{\lfloor \frac{n}{2}\rfloor}$ for all $n\geq2$.
\end{proof}
\begin{Remark}
Since, a path $P_n$ is a graph with fixing number $1$ and a complete
graph $K_n$ is the graph with fixing number $n-1$, hence we can
deduce that for a connected graph $G$ of order $n\geq2$, $1\leq
F_{sum}(D)\leq {n \choose 2}$ and $\frac{2}{n^2-n}\leq F\%(D)\leq
\frac{2}{n}$.

\end{Remark}
Two vertices $u$ and $v$ in a connected graph $G$ of order $n$ are
said to be \emph{antipodal} vertices if $d(u,v)=\frac{n}{2}$.
\begin{Theorem}\label{t34}
Let $G$ be a cycle graph of order $n\geq4$ and let $D$ be a minimum
fixing set of $G$. For each $v\in D$, the fixing share
$f(v;D)=\frac{1}{2}{n \choose 2}$. Moreover, $F_{sum}(D)={n \choose
2}$ and $F\%(D)=\frac{4}{n^2-n}$.
\end{Theorem}
\begin{proof}
Let $D=\{u,v\}$ be a minimum fixing set of $G$ consisting of two
non-antipodal vertices $u$ and $v$, and observe that $|V_s(G)|={n
\choose 2}$. We distinguish two cases.
\begin{enumerate}
\item Case (when $n$ is even):
We notice that $|F(v)|={n\choose 2}-\frac{n-2}{2}$. There are two
types of pairs $(x,y)$ in $F(v)$. Indeed, $(i)$ there are $(\frac
{n-2}{2})$ pairs $(x,y)$ in $F(v)$ such that $|\mathcal{F}(x,y)|=1$
(as $v$ is the sole fixer for $(\frac {n-2}{2})$ pairs in $F(v)$),
and $(ii)$ there are ${n\choose 2}-(n-2)$ remaining pairs $(x,y)$ in
$F(v)$ such that $|\mathcal{F}(x,y)|=2$ (as $u$ and $v$ equally
participate to fix ${n\choose 2}-(n-2)$ remaining pairs in $F(v)$).
Thus, $f(v;D)=\frac{n-2}{2}+\frac{1}{2}[{n \choose
2}-(n-2)]=\frac{1}{2}{n\choose 2}$ for each $v\in D$.

\item Case (when $n$ is odd):
We notice that $|F(v)|={n\choose 2}-\frac{n-1}{2}$. There are two
types of pairs $(x,y)$ in $F(v)$. Indeed, $(i)$ there are $(\frac
{n-1}{2})$ pairs $(x,y)$ in $F(v)$ such that $|\mathcal{F}(x,y)|=1$
(as $v$ is the sole fixer for $(\frac {n-1}{2})$ pairs in $F(v)$),
and $(ii)$ there are ${n\choose 2}-(n-1)$ remaining pairs $(x,y)$ in
$F(v)$ such that $|\mathcal{F}(x,y)|=2$ (as $u$ and $v$ equally
participate to fix the remaining ${n\choose 2}-(n-1)$ pairs in
$F(v)$). Thus, $f(v;D)=\frac{n-1}{2}+\frac{1}{2}[{n \choose
2}-(n-1)]=\frac{1}{2}{n\choose 2}$ for each $v\in D$.
\end{enumerate}

Hence, $F_{sum}(D)={n \choose 2}$ and hence
$F\%(D)=\frac{4}{n^2-n}$.
\end{proof}

\section{Summary}
 In this paper, we have first described some properties of
$fix(u,v)$. Then, we have found bound on the cardinality of the edge
set of the fixing graph of a graph $G$. Also, we have defined fixing
share of a vertex in a fixing set $D$ of a graph $G$ and studied it
for vertices in fixing sets of some common classes of graphs.

\end{document}